\documentclass[12pt]{amsart}

\usepackage{amsfonts}
\usepackage{latexsym}
\usepackage{amsmath}
\usepackage{amssymb}
\usepackage{amsthm}
\usepackage{enumerate} 
\usepackage{mathtools} 
\usepackage[margin=1.2in]{geometry}
\usepackage{mathabx}

\theoremstyle{plain}
\newtheorem{theorem}{Theorem}[section]
\newtheorem{lemma}[theorem]{Lemma}

\newtheorem{corollary}[theorem]{Corollary}

\theoremstyle{definition}

\newcommand\naturals{\mathbb{N}}

\newcommand\reals{\mathbb{R}}
\newcommand\complexes{\mathbb{C}}

\newcommand\norm[1]{\left\lVert#1\right\rVert}
\newcommand\ev[2]{\left<#1,#2\right>}

\DeclareMathOperator{\im}{Im}
\DeclareMathOperator{\re}{Re}

\DeclareMathOperator{\supp}{supp}

\DeclareMathOperator{\interior}{int}

\newcommand\bigoh[1]{O\left(#1\right)}

\newcommand\ds{\mathcal{D}}

\newcommand\es{\mathcal{E}}

\newcommand\sss{\mathcal{S}}

\newcommand\des{\es^{\prime}}

\newcommand\despaceomega[1]{\testspaceomega{\des}}

\newcommand\beurou{\ast}
\newcommand\utestspace[3]{#1^{#2}\left(#3\right)}
\newcommand\utestspacereals[3]{\utestspace{#1}{#2}{\reals^{#3}}}

\newcommand\udspacereals[2]{\utestspacereals{\ds}{#1}{#2}}

\newcommand\laplace[2]{\mathcal{L}\left\{#1 ; #2\right\}}

\numberwithin{equation}{section}

\begin{document}

\title[Multidimensional Tauberian theorem for Laplace transforms]{A multidimensional Tauberian theorem for Laplace transforms of ultradistributions}

\author[L. Neyt]{Lenny Neyt}
\thanks{L. Neyt gratefully acknowledges support by Ghent University, through the BOF-grant 01J11615.}

\author[J. Vindas]{Jasson Vindas}

\thanks {The work of J. Vindas was supported by Ghent University through the BOF-grants 01J11615 and 01J04017.}

\address{Department of Mathematics: Analysis, Logic and Discrete Mathematics\\ Ghent University\\ Krijgslaan 281\\ 9000 Gent\\ Belgium}
\email{lenny.neyt@UGent.be}
\email{jasson.vindas@UGent.be}



\begin{abstract}
We obtain a multidimensional Tauberian theorem for Laplace transforms of Gelfand-Shilov ultradistributions. The result is derived from a Laplace transform characterization of bounded sets in spaces of ultradistributions with supports in a convex acute cone of $\mathbb{R}^{n}$, also established here.
\end{abstract}

\maketitle

\section{Introduction}

In 1976, Vladimirov obtained an important multidimensional generalization of the Hardy-Littlewood-Karamata Tauberian theorem \cite{vladimirov1976}. Multidimensional Tauberian theorems were then systematically investigated by him, Drozhzhinov, and Zav'yalov, and their approach resulted in a powerful Tauberian machinery for multidimensional Laplace transforms
of Schwartz distributions. Such results have been very useful in probability theory \cite{Yakimiv} and mathematical physics  \cite{BerceanuGheorghe,d-z2008,Vladimirov1979}. 
Tauberian theorems for other integral transforms of generalized functions have been extensively studied by several authors as well, see e.g. \cite{d-z1998, d-z2003,pilipovic-stankovic1997,P-V2014,P-V2019}. We refer to the monographs \cite{stevan2011asymptotic,vladimirovbook,vladimirov1988tauberian} for accounts on the subject and its applications; see also the recent survey article \cite{drozhzhinov2016}. 

The aim of this article is to extend the so-called general Tauberian theorem for the dilation group \cite[Chapter 2]{vladimirov1988tauberian} from distributions  to ultradistributions. Our considerations apply to Laplace transforms of elements of $\mathcal{S}'^{\ast}_{\dagger}[\Gamma]$, the space of Gelfand-Shilov ultradistributions with supports in a closed convex acute cone $\Gamma$ of $\mathbb{R}^{n}$ where $\ast$ and $\dagger$ stand for the Beurling and Roumieu cases of weight sequences satisfying mild assumptions  (see Section \ref{preliminaries section} for definitions and notation). In Section \ref{Laplace and bounded sets} we provide characterizations of bounded sets and convergent sequences in $\mathcal{S}'^{\ast}_{\dagger}[\Gamma]$ in terms of Laplace transform growth estimates; interestingly, our approach to the desired Laplace transform characterization is based on a useful convolution average description of bounded sets of $\mathcal{S}'^{\ast}_{\dagger}(\mathbb{R}^{n})$, originally established in \cite{dimovski2016} (cf. \cite{P-P-Vconv}) but improved here by relaxing hypotheses on the weight sequences. Those results are employed in Section \ref{Section: Tauberian theorem} to derive a Tauberian theorem in which the quasiasymptotic behavior of an ultradistribution is deduced from asymptotic properties of its Laplace transform. Finally, as a natural refinement of the main result of Section \ref{Laplace and bounded sets} when the weight sequences and the cone satisfy stronger regularity conditions, we prove in Section  \ref{Section Laplace isomorphism} that the Laplace transform is an isomorphism of locally convex spaces between  $\mathcal{S}'^{\ast}_{\dagger}[\Gamma]$ and a certain space of holomorphic functions on the tube domain $\mathbb{R}^{n}+i\operatorname*{int}\Gamma^{\ast}$, with $\Gamma^{\ast}$ the conjugate cone of $\Gamma$.

\section{Preliminaries}
\label{preliminaries section}
We collect in this section several useful notions that play a role in the
article.
\subsection{Weight sequences} Given a weight sequence $\{ M_{p} \}_{p \in \naturals}$ of positive real numbers, we associate to it the sequences $m_{p} = M_{p} / M_{p - 1}$, $p \geq 1$, and $M_{p}^{*} = M_{p} / p!$. Throughout this article we will often make use of some of the following conditions:
	\begin{description}
		\item[$(M.1)$] $M_{p}^{2} \leq M_{p - 1} M_{p + 1}$ , $p \geq 1$ ; 
		\item[$(M.1)^{*}$] $(M_{p}^{*})^{2} \leq M_{p - 1}^{*} M_{p + 1}^{*} , p \geq1$;
		\item[$(M.2)'$] $M_{p + 1} \leq A H^{p} M_{p}$, $p \in \naturals$, for certain constants $A, H\geq 1$;
		\item[$(M.2)$] $M_{p+q} \leq A H^{p+q} M_{p} M_{q}$, $p,q \in \naturals$, for certain constants $A, H\geq 1$;
	        \item[$(M.3)'$] $\sum_{p = 1}^{\infty} 1/m_p < \infty$;
		\item[$(M.3)$] $ \sum_{p=q}^{\infty} 1/m_p \leq c_0 q /m_{q}$, $q\geq 1$, for a certain constant $c_0$.
	\end{description}
The meaning of all these conditions is very well explained in \cite{ultradistributions1}. Whenever we consider weight sequences, we assume they satisfy  at least $(M.1)$.
For multi-indices $\alpha \in \naturals^{n}$, we will simply denote $M_{|\alpha|}$ by $M_{\alpha}$.  As usual the relation $M_{p} \prec N_{p}$ between two such sequences means that for any $h > 0$ there is an $L = L_{h} > 0$ for which $M_{p} \leq L h^{p} N_{p}$, $p \in \naturals$. The associated function of the sequence $M_{p}$ is given by
	\begin{equation}
		\label{eq:assfunc}
		M(t) := \sup_{p \in \naturals} \log \frac{t^{p} M_{0}}{M_{p}}, \qquad t>0,
	\end{equation}
and $M(0)=0$.  
It increases faster than $\log t$ as $t\to\infty$ (cf. \cite[p.~48]{ultradistributions1}). The associated function of the sequence $M_{p}^{*}$ will be denoted by $M^{*}(t)$. 

Throughout this text we shall often exploit the following bounds:
			\begin{itemize}
				\item If $M_{p}$ satisfies $(M.2)'$, then for any $k > 0$
					\begin{equation}
						\label{eq:M2prime}
						M(t) - M(k t) \leq - \frac{\log(t / A) \log k}{\log H} , \qquad t > 0 . 
					\end{equation}
				\item $M_{p}$ satisfies $(M.2)$ if and only if
					\begin{equation}
						\label{eq:M2}
						2 M(t) \leq M(H t) + \log(A M_{0}) . 
					\end{equation}
				\item If $M_p$ satisfies $(M.1)^{\ast}$, we have, for some $A'>0$,
					\begin{equation}
						\label{eq:M*bound}
						M^{*}\left(\frac{t}{4(m_{1} + 1) M(t)}\right) \leq M(t) + A' , \qquad t \geq m_{1} + 1 .
					\end{equation}
			\end{itemize}
Indeed the first and second statements are \cite[Proposition~3.4]{ultradistributions1} and \cite[Proposition~3.6]{ultradistributions1}, while the third one is shown in \cite[Lemma~5.2.5, p.~96]{carmichael2007boundary}.
 We shall also consider the following two sets
	\begin{align*} 
		\mathfrak{R}^{(M_{p})}   &:= \{ (\ell_{p})_{p \in \naturals^{+}} : \ell_{p} = \ell \text{ for some } \ell > 0 \} , \\
		\mathfrak{R}^{\{M_{p}\}} &:= \{ (\ell_{p})_{p \in \naturals^{+}} : \ell_{p} \nearrow \infty \text{ and } \ell_{p} > 0 , \forall p \in \naturals  \} , 
	\end{align*}
and use $\mathfrak{R}^{\beurou}$ as a common notation. Naturally, these two sets do not depend on $M_p$ at all, but it will be very convenient for us to make a notational distinction between the Beurling and Roumieu case of a weight sequence when dealing with ultradistributions. For any $(\ell_{p}) \in \mathfrak{R}^{\ast}$, we write $L_p=\prod_{j = 1}^{p} \ell_{j}$ and denote the associated function of $M_{p}L_p$ as $M_{\ell_{p}}(t)$. The reader should keep in mind that whenever $(M.1)$ holds one has the ensuing useful assertions \cite[Lemma~4.5, p. 417]{D-V-V2018} on the growth of a function $g:[0,\infty)\to[0,\infty)$
\begin{equation} 
\label{eqequivalencegrowth1}
\forall h>0:\: g(t) = O(e^{M(ht)}) \iff \exists (\ell_{p})\in \mathfrak{R}^{\left\{M_p\right\}}:\: g(t) = O(e^{M_{\ell_p}(t)})
\end{equation}
and 
\begin{equation} 
\label{eqequivalencegrowth2}
\forall (\ell_{p})\in\mathfrak{R}^{\left\{M_p\right\}}:\: g(t) = O(e^{-M_{\ell_p}(t)}) \iff \exists h>0:\: g(t) = O(e^{-M(ht)}).
\end{equation}
 
It is also important to point out that if $M_{p}$ satisfies $(M.2)$ or $(M.2)'$,  then for any given $(\ell_p)\in\mathfrak{R}^{\beurou}$ one can always find a $(k_{p}) \in \mathfrak{R}^{\beurou}$ such that $k_{p} \leq \ell_{p}$, $\forall p \in \naturals$, and $M_{p}K_p$ satisfies the same condition as $M_{p}$. For the $(M_{p})$-case this is trivial, whereas the assertion for the $\{M_{p}\}$-case directly follows from \cite[Lemma 2.3]{BojangLaplaceUltradistr}. 
	
\subsection{Ultradistributions} We now introduce the spaces of test functions and ultradistributions that we need in this work. Let $M_p$ and $N_p$ be two weight sequences.  We will always assume that the sequence $M_p$ satisfies the conditions $(M.1)$, $(M.2)'$, and $(M.3)'$. On the other hand, our assumptions on $N_p$ are $(M.1)^{\ast}$ and $(M.2)$. Furthermore, whenever considering the Beurling case we assume in addition that $N_p$ fulfills
\begin{description}
\item [$(NA)$] $ p! \prec N_{p}$.
\end{description} 
Note that these assumptions ensure that $N_{\ell_{p}}(t)=o(t)$ \cite[Lemma 3.8 and Lemma~3.10, p.~52--53]{ultradistributions1}, $N^{\ast}_{\ell_p}(t)<\infty$ for all $t\geq 0$, and $N^{\ast}_{\ell_{p}}(t)\to\infty$ as $t\to\infty$ for any sequence $(\ell_p)\in \mathfrak{R}^{\beurou}$. If stronger assumptions on the weight sequences are needed, this will be explicitly stated in the corresponding statement.
 
Let us now define Gelfand-Shilov spaces with respect to the sequences $M_p$ and $N_p$. We use the common notation $\ast=(M_p),\{M_p\}$ and $\dagger=(N_p),\{N_p\}$ for the Beurling and Roumieu cases of the weight sequences.  For any $(a_{p}), (b_{p}) \in \mathfrak{R}^{\beurou}$ we consider the Banach space of all $\varphi \in C^{\infty}(\mathbb{R}^{n})$ such that
	\begin{equation} 
		\label{eq:norm}
		\norm{\varphi}_{(a_{p}),(b_p)} = \sup_{\alpha, \beta \in \naturals^{n}}\sup_{t\in\mathbb{R}^{n}} \frac{|t^{\beta} \varphi^{(\alpha)}(t)|}{A_{\alpha} M_{\alpha} B_{\beta} N_{\beta}} , 
	\end{equation}
and denote it by $\sss^{M_p, (a_{p})}_{N_p, (b_{p})}(\mathbb{R}^{n})$. Then, we define the test function spaces
	\[ \sss^{\beurou}_{\dagger}(\mathbb{R}^{n}) := \varprojlim_{(a_{p}), (b_{p}) \in \mathfrak{R}^{*}} \sss^{M_p,(a_{p})}_{N_p, (b_{p})}(\mathbb{R}^{n}) , \]
and consider their duals, the ultradistribution spaces $\sss^{\prime \beurou}_{\dagger}(\mathbb{R}^{n})$ \cite{carmichael2007boundary, P-P-Vconv}.  As classically done, the Roumieu type space $\sss^{\beurou}_{\dagger}(\mathbb{R}^{n})$ could have also be introduced via an inductive limit, and that definition coincides (algebraically and topologically) with the projective description given here (see e.g. \cite{D-V-Roumieuind}).  One has that $\mathcal{S}^{(M_p)}_{(N_p)}(\mathbb{R}^{n})$ is an $(FS)$-space, while $\mathcal{S}^{\{M_p\}}_{\{N_p\}}(\mathbb{R}^{n})$ is $(DFS)$. 

The subspace of $\sss^{\beurou}_{\dagger}(\mathbb{R}^{n})$ consisting of compactly supported elements is denoted as usual as $\mathcal{D}^{\ast}(\mathbb{R}^{n})$ (it is non-trivial \cite{ultradistributions1} because of $(M.3)'$) and we write $\mathcal{D}^{\beurou}_{K}$ for those elements of $\mathcal{D}^{\ast}_{K}$ whose support is contained in a given compact subset $K\subset\mathbb{R}^{n}$. Similarly, $\mathcal{E}^{\ast}(\mathbb{R}^{n})$ stands for the space of all $\ast$-ultradifferentiable functions on $\mathbb{R}^{n}$. These spaces are topologized in the canonical way \cite{ultradistributions1, ultradistributions3}.

\subsection{Laplace transform}  Throughout the article
 $\Gamma \subseteq \reals^{n}$ stands for a (non-empty) closed, convex, and acute cone (with vertex at the origin).
Acute means that its conjugate cone, 
	\[ \Gamma^{*} := \{ y \in \reals^{n} :\: y \cdot u \geq 0,\: \forall u \in \Gamma \} , \]
has non-empty interior and we set $C = \interior \Gamma^{*}$. Note that $\Gamma^{\ast \ast}=\Gamma$. The distance of a point $y\in\mathbb{R}^{n}$ to the boundary of $C$ is denoted as
	\[ \Delta_{C}(y) := d(y, \partial C).\]
We will often make use of the estimate
	 (cf. \cite[p. 61]{vladimirovbook})
			\begin{equation} 
				\label{eq:dotestimate}
				y \cdot u \geq \Delta_{C}(y) |u| , \qquad \forall u \in \Gamma, y \in C. 
			\end{equation}
The tube domain $T^{C}$ with base $C$ is the set
	\[ T^{C} := \reals^{n} + i C \subseteq \complexes^{n} . \]
For any $\varepsilon > 0$, we denote by $\Gamma^{\varepsilon}$ the open set $\Gamma + B(0, \varepsilon)$.  We define
\[\mathcal{S}^{\prime\beurou}_{\dagger}[\Gamma]=\{f\in\mathcal{S}^{\prime\beurou}_{\dagger}(\mathbb{R}^{n}):\: \operatorname*{supp}f \subseteq \Gamma\};
\]
it is a closed subspace of $\mathcal{S}^{\prime\beurou}_{\dagger}(\mathbb{R}^{n})$.

Let $\eta : \reals^{n} \rightarrow \reals$ be a function such that $\eta(\xi) = 1$ for $\xi \in \Gamma$ and $\eta(\xi) e^{i z \cdot \xi} \in \sss^{\beurou}_{\dagger}(\reals^{n})$ for any $z \in T^{C}$.  The \emph{Laplace transform} of $f\in\mathcal{S}^{\prime\beurou}_{\dagger}[\Gamma]$ is then the holomorphic function
			%
			\[
				\laplace{f}{z} := \ev{f(\xi)}{\eta(\xi) e^{i z \cdot \xi}} , \qquad z \in T^{C} .
			\]
One can always find such an $\eta$ (see e.g. Lemma \ref{l:Mp-Gamma-func-existence} below) and the definition of the Laplace transform does not depend on this function. We write $z=x+iy$ for complex variables.

\section{Laplace transform characterization of bounded sets of  $\sss^{\prime\beurou}_{\dagger}[\Gamma]$}
\label{Laplace and bounded sets}
In this section we shall characterize those subsets of $\sss^{\prime\beurou}_{\dagger}[\Gamma]$  that are bounded (with respect to the relative topology inherited from $\sss^{\prime\beurou}_{\dagger}(\reals^{n})$) via bounds on the Laplace transforms of their elements.  The following theorem is our main result in this section.

\begin{theorem}
		\label{t:boundednessLaplace}
		Let $B \subseteq \sss^{\prime\beurou}_{\dagger}[\Gamma]$. 
\begin{enumerate}
\item [(i)] If $B$ is a bounded set, then, there is $(\ell_{p})\in\mathfrak{R}^{\beurou}$ for which, given any $\varepsilon>0$, there is $L = L_{\varepsilon} > 0$ such that  for all $f \in B$
			\begin{equation}
				\label{eq:boundednessLaplace1}
				\left| \laplace{f}{z} \right| \leq L \exp\left( \varepsilon |\im z| + M_{\ell_{p}}(|z|) + N_{\ell_{p}}^{*}\left(\frac{1}{\Delta_{C}(\im  z)} \right) \right) , \qquad z \in T^{C} . 
			\end{equation}
\item [(ii)] Conversely, suppose there are $\omega\in C$, $\sigma_0>0$, $L = L_{B} > 0$, and $(\ell_{p})\in\mathfrak{R}^{\beurou}$  such that
\begin{equation}
				\label{eq:boundednessLaplace2}
				\left| \laplace{f}{x+i\sigma \omega} \right| \leq L \exp\left(M_{\ell_{p}}(|x|) + N_{\ell_{p}}^{*}\left(\frac{1}{\sigma} \right) \right),
			\end{equation}
for all $f\in B$, $x\in\mathbb{R}^{n}$, and $\sigma\in (0,\sigma_0]$, then $B$ is a bounded subset of $\sss^{\prime\beurou}_{\dagger}[\Gamma]$.
\end{enumerate}
	\end{theorem}

Before proving Theorem \ref{t:boundednessLaplace}, let us discuss an important consequence. Namely, we shall derive from it a characterization of convergent sequences of $\sss^{\prime\beurou}_{\dagger}[\Gamma]$. Notice first that if a sequence $f_{k}\to g$ in  $\sss^{\prime\beurou}_{\dagger}(\mathbb{R}^{n})$ and $\operatorname*{supp} f_{k}\subseteq \Gamma$ for each $k$, one easily shows that 
$$
\lim_{k\to\infty}\laplace{f_{k}}{z}=\laplace{g}{z},
$$
and this limit holds uniformly for $z$ in compact subsets of $T^{C}$; furthermore, by Theorem \ref{t:boundednessLaplace}, the Laplace transforms of the $f_k$ satisfy bounds of the form \eqref{eq:boundednessLaplace1} uniformly in $k$. The converse also holds. In fact, the next result might be interpreted as a sort of Tauberian theorem.

	\begin{corollary}
		\label{t:generalTauberianTheorem}
		Let $(f_{k})_{k \in \naturals}$ be a sequence in $\sss^{\prime\beurou}_{\dagger}[\Gamma]$. Suppose that there is a non-empty open subset $\Omega\subseteq C$ such that for each $y\in\Omega$ the limit
	\begin{equation}
						\label{eq:generalTauberianTheorem-Laplaceconv}
						\lim_{k \rightarrow \infty} \laplace{f_{k}}{iy}
						\end{equation}	
exists. If there are $\omega\in C$, $\sigma_0>0$, and $(\ell_{p})\in\mathfrak{R}^{\beurou}$  such that
\begin{equation}
				\label{eq:generalTauberianTheorem-Bound}
			\sup_{k\in\mathbb{N}, \: x\in\mathbb{R}^{n},\: \sigma\in (0,\sigma_0]}	 \exp\left(-M_{\ell_{p}}(|x|) - N_{\ell_{p}}^{*}\left(\frac{1}{\sigma} \right) \right)\left| \laplace{f_k}{x+i\sigma \omega} \right| <\infty
			\end{equation}	
then
			\begin{equation}
				\label{eq:generalTauberianTheorem-conv}
				\lim_{k \rightarrow \infty} f_{k} = g  \quad \mbox{in } \sss^{\prime\beurou}_{\dagger}[\Gamma],
			\end{equation}
for some $g\in\sss^{\prime\beurou}_{\dagger}[\Gamma]$. In particular, the limit \eqref{eq:generalTauberianTheorem-Laplaceconv} is given by $\laplace{g}{iy}$.

\end{corollary}
	
	\begin{proof} Notice first that if two subsequences converge, respectively, to ultradistributions $g$ and $h$, the limits \eqref{eq:generalTauberianTheorem-Laplaceconv} tell us $\laplace{g}{iy}=\laplace{h}{iy}$ for all $y\in \Omega$. By uniqueness of holomorphic functions and the injectivity of the Laplace transform (which follows from that of the Fourier transform), we conclude $g=h$.  It therefore suffices to show that every arbitrary subsequence of the $f_k$ possesses a convergent subsequence in $\sss^{\prime\beurou}_{\dagger}[\Gamma]$, but this follows from the fact that $\sss^{\prime\beurou}_{\dagger}[\Gamma]$ is Montel
	because, in view of Theorem \ref{t:boundednessLaplace}, the estimate \eqref{eq:generalTauberianTheorem-Bound} is equivalent to $\{f_{k}:\: k\in\mathbb{N}\}$ being bounded in $\sss^{\prime\beurou}_{\dagger}[\Gamma]$ (and hence relatively compact). 	\end{proof}

We shall prove Theorem \ref{t:boundednessLaplace} using several lemmas. For (i), we need the following concept. A family $\{\eta_{\varepsilon}\}_{\varepsilon > 0}$ of non-negative smooth functions $\eta_{\varepsilon} : \reals^{n} \rightarrow [0,\infty)$ is called a $\beurou$-$\Gamma$-mollifier if for every $\varepsilon > 0$ the ensuing conditions hold
			\begin{enumerate}[(a)]
				\item $\eta_{\varepsilon}(\xi) = 1$ for $\xi \in \Gamma^{\varepsilon}$ while $\eta_{\varepsilon}(\xi) = 0$ for $\xi \notin \Gamma^{2 \varepsilon}$;
				\item for every $(\ell_{p}) \in \mathfrak{R}^{\beurou}$ there is a constant $H_{\ell_{p}, \varepsilon} > 0$ such that
					\begin{equation}
						\label{eq:Mp-Gamma-func-bound}
						\left| \eta^{(\alpha)}_{\varepsilon}(\xi) \right| \leq H_{\ell_{p}, \varepsilon} L_{\alpha} M_{\alpha} , \qquad \forall \xi \in \reals^{n}, \forall \alpha \in \naturals^{n} . 
					\end{equation}
			\end{enumerate}

	\begin{lemma}
		\label{l:Mp-Gamma-func-existence}
		There are $\beurou$-$\Gamma$-mollifiers.
	\end{lemma}
	
	\begin{proof} 	
The existence of such functions is guaranteed by non-quasianalyticity. 
		Take any non-negative $\varphi \in \udspacereals{\beurou}{n}$ such that $\supp \varphi \subset B(0, 1/2)$ and $\int_{\mathbb{R}^{n}} \varphi(\xi)d\xi = 1$. Set $\varphi_{\varepsilon}(\xi) := \varepsilon^{-n} \varphi(\xi / \varepsilon)$ and let $\chi_{\Gamma^{3\varepsilon/2}}$ be the characteristic function of $\Gamma^{\frac{3}{2} \varepsilon}$. Taking $\eta_{\varepsilon} = \varphi_{\varepsilon} * \chi_{\Gamma^{3\varepsilon/2}}$, one easily verifies that $\{\eta_{\varepsilon}\}_{\varepsilon > 0}$ is a $\beurou$-$\Gamma$-mollifier.
	
	\end{proof}

	\begin{lemma}
		\label{l:expbound}
		Let $(a_{p}), (b_{p}) \in \mathfrak{R}^{*}$ and $\{\eta_{\varepsilon}\}_{\varepsilon > 0}$ be a $\beurou$-$\Gamma$-mollifier. Then there is $(\ell_{p}) \in \mathfrak{R}^{*}$ such that, for any $\varepsilon > 0$, we have
			\begin{equation*}
				\norm{\eta_{\varepsilon}(\xi) e^{i z \cdot \xi}}_{(a_{p}), (b_{p})} \leq H_{\ell_{p}, \varepsilon} \exp\left( 4 \varepsilon |\im z| + M_{\ell_{p}}(|z|) + N_{\ell_{p}}^{*}\left(\frac{1}{\Delta_{C}(\im z)}\right) \right), \qquad z \in T^{C}.
			\end{equation*}
			In particular, we have  $\eta_{\varepsilon}(\xi) e^{i z \cdot \xi} \in \sss^{M_p, (a_{p})}_{N_p, (b_{p})}(\mathbb{R}^{n})$ for all $z\in T^{C}$. 
				\end{lemma}
	
	\begin{proof} We only employ here the assumptions $(M.1)$ and $(M.3)'$ for the sequence $M_p$. 
		Set $\ell'_{p} := \min\{a_{p}, b_{p}\}$. Due to the support assumption on $\eta_{\varepsilon}$, we may assume below that $\xi \in \Gamma^{2 \varepsilon}$. Then for any $z \in T^{C}$, $\alpha, \beta \in \naturals^{n}$, we have
			\begin{align*}
				\frac{\left| \xi^{\beta} \frac{\displaystyle\partial^{\alpha}}{\displaystyle\partial \xi^{\alpha}} \left( \eta_{\varepsilon}(\xi) e^{iz \cdot \xi} \right) \right|}{A_{\alpha} M_{\alpha} B_{\beta} N_{\beta}}
				&
			 \leq \frac{|\xi|^{\beta} e^{- y \cdot \xi}}{L'_{\beta} N_{\beta}} 2^{-|\alpha|}\sum_{0 \leq \alpha' \leq \alpha} {\alpha \choose \alpha' }\frac{(2 |z|)^{|\alpha'|}}{L'_{\alpha'} M_{\alpha'}} \left( \frac{2^{|\alpha - \alpha'|}}{L'_{\alpha - \alpha'} M_{\alpha - \alpha'}} \left| \eta_{\varepsilon}^{(\alpha - \alpha')}(\xi) \right| \right) 
				\\
				&
			 \leq H_{\ell_{p}, \varepsilon} e^{M_{\ell_{p}}(|z|)} \frac{|\xi|^{\beta} e^{-y \cdot \xi}}{L_{\beta} N_{\beta}} ,
			\end{align*}
		where we have set $\ell_{p} := \ell'_{p} / 2$. Now $\xi = u + v$ for certain $u \in \Gamma$ and $v \in B(0, 2 \varepsilon)$, so that by the Cauchy-Schwarz inequality
			\begin{align*}
				\frac{|\xi|^{\beta} e^{-y \cdot \xi}}{L_{\beta} N_{\beta}} &\leq \frac{(|u| + 2 \varepsilon)^{\beta} e^{-y \cdot u} e^{-y \cdot v}}{L_{\beta} N_{\beta}} \leq \frac{(|u| + 2 \varepsilon)^{|\beta|} e^{-\Delta_{C}(y) |u|} e^{2 \varepsilon |y|}}{L_{\beta} N_{\beta}} \\
				&\leq \frac{\left(\frac{1}{\Delta_{C}(y)}\right)^{\beta} \left(\frac{|\beta|}{e}\right)^{\beta}}{L_{\beta} N_{\beta}} e^{2 \varepsilon \Delta_{C}(y) + 2 \varepsilon |y|} \leq \exp\left( N_{\ell_{p}}^{*}\left(\frac{1}{\Delta_{C}(y)}\right) + 4 \varepsilon |y| \right) , 
			\end{align*}
		where we have used \eqref{eq:dotestimate}			and the elementary inequality $m^m\leq e^m m!$.
	\end{proof}
	
In preparation for the proof of part (ii), we first need to extend \cite[Proposition 3.1]{dimovski2016} (cf. \cite[Lemma 2.7]{P-P-Vconv}) by relaxing assumptions on the weight sequences. This provides a useful convolution characterization of bounded sets in $\sss^{\prime \beurou}_{\dagger}(\reals^{n})$. Our approach to this convolution characterization employs the short-time Fourier transform (STFT) in the context of ultradistributions \cite{GrochenigZimmermann, D-V-indlimultra} and is inspired by the method from \cite{K-P-S-V2016}.
Given an ultradistribution and a window $\psi$ (a test function), the STFT of $f$ with respect to $\psi$ is given by the smooth function
$$
		V_{\psi} f (x, \xi) = \langle f(t), \overline{\psi(t - x)} e^{- 2 \pi i \xi \cdot t}\rangle , \qquad (x, \xi) \in \reals^{2n} . 
$$
	\begin{lemma}
		\label{p:boundednessconvolution}
		A subset  $B\subset\sss ^{\prime\beurou}_{\dagger}(\reals^{n})$ is bounded if and only if there exists $(\ell_{p}) \in \mathfrak{R}^{\beurou}$ such that
			\begin{equation}
				\label{eq:boundednessconvolution}
				\sup_{f \in B,\: x \in \reals^{n}} e^{-N_{\ell_{p}}(|x|)} \left| (f * \psi)(x) \right| < \infty  , \qquad \forall \psi \in \udspacereals{\beurou}{n} . 
			\end{equation}
	\end{lemma}
	
	\begin{proof} We only make use here of the assumptions $(M.1)$ and $(M.2)'$ on $N_p$.	The	necessity is easily obtained via the norms (\ref{eq:norm}). Hence suppose that (\ref{eq:boundednessconvolution}) holds for some $(\ell_{p}) \in \mathfrak{R}^{\beurou}$. We may assume the sequence $L_{p} N_{p}$ satisfies $(M.2)'$.	 
	We consider the weighted Banach space $X=\{g\in C(\mathbb{R}^{n}):\: g(\xi)=O(\exp(N_{\ell_{p}}(|\xi|)))\}$ and fix a compact set $K\subset \mathbb{R}^{n}$ with non-empty interior. 
	
The assumption (\ref{eq:boundednessconvolution})  implies that for each $f\in B$ the mapping $L_{f}: \varphi \mapsto f\ast \varphi $ is continuous from $\mathcal{D}^{\ast}(\mathbb{R}^{n})$ into $X$, so that in particular, in view of the Banach-Steinhaus theorem, $\tilde{B}=\{(L_f)_{|\mathcal{D}^{\ast}_{K}}: \: f\in B\}$ is an equicontinuous subset of $L_{b}(\mathcal{D}^{\ast}_{K}, X)$. This implies that there is $(h_p)\in\mathfrak{R}^{\ast}$ such that $\tilde{B}\subset L_{b}(\mathcal{D}^{M_p}_{K,(h_{p})}, X)$ and it is equicontinuous there, where $\mathcal{D}^{M_p}_{K,(h_{p})}=\{\psi\in \mathcal{D}_{K}:\: \sup_{x\in K,\: \alpha\in\mathbb{N}^{n}}|\psi^{(\alpha)}(x)|/(H_\alpha M_\alpha)<\infty \}$.
Fix $\psi\in \mathcal{D}^{(M_p)}_{K}$ with $\|\psi\|_{L^{2}(\mathbb{R}^{n})}=1$. Since $\{e^{-M_{h_{p}}(4 \pi |\xi|)} e^{2 \pi i \xi \cdot} \widecheck{\overline{\psi}} : \xi \in \reals^{n} \}$ is a bounded family in $\mathcal{D}^{M_p}_{K, (h_p)}$, we conclude that,  for some $C_{B} > 0$, independent of $f \in B$,
			\[ \left| V_{\psi} f(x, \xi) \right| = \left| e^{-2 \pi i \xi \cdot x} \left( f * (e^{2 \pi i \xi \cdot} \widecheck{\overline{\psi}}) \right)(x) \right| \leq C_{B} \exp\left(N_{\ell_{p}}(|x|) + M_{h_{p}}(4 \pi |\xi|)\right) . \]	
			
	On the other hand, let now $\varphi \in \sss_{\dagger}^{\beurou}(\reals^{n})$.	 For any $(\ell'_{p}) \in \mathfrak{R}^{\ast}$ it follows from \cite[Proposition 1]{D-V-Roumieuind} that  there is some $C_{\varphi} > 0$ such that
			\[ \left| V_{\overline{\psi}} \varphi(x, -\xi) \right| \leq C_{\varphi} \exp\left( - N_{\ell'_{p}}(|x|) - M_{\ell'_{p}}(|\xi|) \right) . \]
Moreover, according to the desingularization formula\footnote{This is stated in \cite{D-V-indlimultra} under the assumptions $(M.1)$ and $(M.2)$, but one can relax $(M.2)$ to $(M.2)'$ using the continuity result \cite[Proposition 1]{D-V-Roumieuind} and adapting the arguments given in \cite[Section 3]{K-P-S-V2016} or \cite{D-V-indlimultra}. } 	for the STFT \cite[Eq. (2.6)]{D-V-indlimultra},
$$
\ev{f}{\varphi} = \int \int_{\reals^{2n}} V_{\psi} f(x, \xi) V_{\overline{\psi}} \varphi(x, -\xi) dx d\xi.
$$
Let $h > 0$ be such that $\log h / \log H \geq n + 1$ (with $H$ the corresponding constant occurring in $(M.2)'$ for $L_p N_p$ and $H_pM_p$) and set $\ell'_{p} := h^{-1} \min(\ell_{p}, (4\pi )^{-1} h_{p})$, then applying (\ref{eq:M2prime}) one gets
			\begin{align*} 
			\sup_{f\in B}\left| \ev{f}{\varphi} \right| 
							       \leq C_{B} C_{\varphi} \int_{\reals^{n}} e^{M_{h_{p}}(4 \pi |\xi|) - M_{\ell'_{p}}(|\xi|)} d\xi \int_{\reals^{n}} e^{N_{\ell_{p}}(|x|) - N_{\ell'_{p}}(|x|)} dx < \infty, 
			\end{align*}
which concludes the proof of the sufficiency.
	\end{proof}

	We are now ready to present a proof of Theorem \ref{t:boundednessLaplace}.
	
	\begin{proof}[Proof of Theorem \ref{t:boundednessLaplace}]
		Suppose $B \subseteq \sss^{ \prime\beurou}_{\dagger}[\Gamma]$ is bounded in $\sss^{\prime\beurou}_{\dagger}(\reals^{n})$. By equicontinuity, there are certain $(a_{p}), (b_{p}) \in \mathfrak{R}^{\beurou}$ such that $B \subseteq \left( \sss^{M_p, (a_{p})}_{N_p, (b_{p})}(\reals^{n}) \right)^{\prime}$ and it is bounded there. Then, (\ref{eq:boundednessLaplace1}) follows directly from Lemma \ref{l:expbound} (in particular, one does not employ $(M.2)$ for $N_p$ in this implication).

		We now show that (\ref{eq:boundednessLaplace2}) is sufficient to guarantee boundedness. We are going to do this employing Lemma \ref{p:boundednessconvolution}. We may assume that $(\ell_p)$ is such that $L_p M_p$ satisfies $(M.2)'$ and $L_p N^{\ast}_p$ fulfills $(M.2)$ (the constants occurring in these conditions are denoted by $A$ and $H$ below). We may also suppose that $|\omega|=1$. Fix $\varphi \in \udspacereals{\beurou}{n}$. Find $R>0$ such that $\supp \varphi\subset B(0,R)$. We keep $f \in B$. Take a bounded function $\gamma : \reals^{n} \rightarrow (0,\sigma_0]$, which will be specified later. Inverting the Laplace transform of $f\ast \varphi$, 
			\[ (f * \varphi)(t) = \frac{1}{(2 \pi)^{n}} \int_{\reals^{n}+i \gamma(t)\omega} e^{-i z \cdot t} \laplace{f}{z} \laplace{\varphi}{z} dz . \]
		By \cite[Lemma~3.3, p.~49]{ultradistributions1} and \eqref{eqequivalencegrowth2}, we have that for any $(h_p)\in \mathfrak{R}^{\ast}$ 
			\[ \left| \laplace{\varphi}{x+i\gamma(t)\omega} \right| \leq L_{\varphi} \exp\left( - M_{h_p}\left(|x|\right) + R \gamma(t) \right), \qquad x\in \mathbb{R}^{n}. \]
Choose $h > 0$ such that $\log h \geq (n + 1) \log H$. Taking $h_p=\ell_{p}/h$, the condition $(M.2)'$ in the form of estimate (\ref{eq:M2prime}) yields
			\[ M_{\ell_p}(|x|) - M_{h_p}\left(|x|\right)= M_{\ell_p}(|x|) - M_{\ell_p}\left(h |x|\right)\leq  - (n + 1) \log(|x| / A)  , \]
		whence we infer the exponential function of this expression is integrable on $\mathbb{R}^{n}$. Let $d=\Delta_{C}(\omega)$. Employing (\ref{eq:boundednessLaplace2}) we then obtain
			\begin{align*}
				&\left| \int_{\reals^{n}+i \gamma(t)\omega} e^{-i z \cdot t} \laplace{f}{z} \laplace{\varphi}{z} dz \right|
				 \\
				&\leq L_{B} L_{\varphi} \exp\left( \gamma(t) (\omega\cdot t) + N_{\ell_{p}}^{*}\left(\frac{1}{d\gamma(t)}\right) + R\gamma(t)\right) \int_{\reals^{n}} e^{M_{\ell_p}(|x|) - M_{h_p}\left(|x|\right)} dx \\
				&\leq L \exp\left( N^{*}_{\ell_p}\left(\frac{1}{d\gamma(t)}\right) + |t|\gamma(t)\right) ,
			\end{align*}
		for some $L > 0$. Note that $L_p N_p$ satisfies $(M.1)^{\ast}$, so that (\ref{eq:M*bound}) holds for it. Also, since $N_{\ell_p}(t)=o(t)$, there is a sufficiently large $r_0$ such that 
		$$
		\frac{4(n_{1} \ell_1+ 1) N_{\ell_p}(|t|)}{d |t|}\leq \sigma_0 \qquad \mbox{for }|t|>r_{0}.
		$$
		Set $r=\max\{r_0, n_1\ell_1+1\}$,
		 we then define
			\[ \gamma(t) = \begin{dcases} \sigma_0, & |t| < r , \\ \frac{4(n_{1} \ell_1+ 1) N_{\ell_p}(|t|)}{d |t|} , &  |t| \geq r . \end{dcases} \]
		For $|t| < r$ obviously 
			\[  \exp\left( N^{*}_{\ell_p}\left(\frac{1}{d\gamma(t)}\right) + |t|\gamma(t)\right)\leq \exp\left( r\sigma_0  + N^{*}_{\ell_p}\left(\frac{1}{\sigma_0d}\right) \right) . \] 
		If $|t| \geq r$, the inequality (\ref{eq:M*bound}) yields
			\begin{align*}
				&\exp\left( N^{*}_{\ell_p}\left(\frac{1}{d\gamma(t)}\right) + |t|\gamma(t)\right) \\
				&\leq \exp\left( N^{*}_{\ell_p}\left(\frac{|t|}{4(n_{1}\ell_1 + 1) N_{\ell_p}( |t|)}\right) +\frac{4(n_{1}\ell_1 + 1)}{d} N_{\ell_p}(|t|) \right) \\
				&\leq \exp\left( 2^{k} N_{\ell_p}( |t|) + A' \right)
			\end{align*}
		for some $A' > 0$ and  $k= \lceil \log_{2} (1+4(n_{1}\ell_1 + 1)/d) \rceil$. By repeated application of (\ref{eq:M2}) for $N_{\ell_p}$, one obtains
			\[ \exp\left( 2^{k} N_{\ell_p}(|t|)\right) \leq \exp(N_{\ell_p}\left( H^{k} |t| \right) + A'') , \]
		for some $A''' > 0$. Let $a_{p} = \ell_pH^{-pk}$. Summing up, we have shown that 
\[		\sup_{f \in B,\: t \in \reals^{n}} e^{-N_{a_{p}}(|t|)} \left| (f * \varphi)(t) \right| < \infty .\]
Since $\varphi$ was arbitrary, Lemma \ref{p:boundednessconvolution} applies to conclude that $B$ is bounded.

	\end{proof}

\section{The Tauberian theorem}
\label{Section: Tauberian theorem}

We shall now use our results from the previous section to generalize the Drozhzhinov-Vladimirov-Zav'yalov multidimensional Tauberian theorem for Laplace transforms \cite{vladimirovbook,vladimirov1988tauberian} from distributions to ultradistributions.
Our goal is to devise a Laplace transform criterion for the so-called quasiasymptotics. 

The quasiasymptotic behavior was originally introduced by Zav'yalov \cite{Zavyalov} for distributions, but the definition of this concept naturally extends to ultradistributions or other duals \cite{structquasiultra, 10.2307/44095755,stevan2011asymptotic} as follows. Assume that $\mathcal{X}$ is a (barreled) locally convex space of test functions on $\mathbb{R}^{n}$ provided with a continuous action of dilations.  A generalized function $f \in \mathcal{X}'$ is said to have \textit{quasiasymptotic behavior} (at infinity) with respect to a (measurable) function $\rho : \reals_{+} \rightarrow \reals_{+}$ if there is $g \in \mathcal{X}'$ such that
	\begin{equation}
		\label{eq:quasiasymptotic}
		f(\lambda \xi) \sim \rho(\lambda) g(\xi) , \quad \text{ as } \lambda \rightarrow \infty
	\end{equation}
in $\mathcal{X}'$, that is, if for each test function $\varphi\in\mathcal{X}$
$$
\langle f(\lambda\xi),\varphi(\xi)\rangle=\lambda^{-n}\langle f(\lambda\xi),\varphi(\xi/\lambda)\rangle\sim \rho(\lambda)\langle g(\xi),\varphi(\xi) \rangle.
$$
The generalized function $g$ must be homogenous of some degree $\alpha\in\mathbb{R}$ and, if $g\neq0$, the function $\rho$ must be regularly varying \cite{bingham1989regular} of index $\alpha\in\mathbb{R}$, namely,
$$
\lim_{\lambda\to\infty} \frac{\rho(\lambda a)}{\rho(\lambda)}=a^{\alpha}, \qquad \mbox{for each } a>0.
$$
Note that since we are only interested in its terminal behavior, one may assume \cite{bingham1989regular} without any loss of generality that the regularly varying function $\rho$ is continuous on $[0,\infty)$. We are exclusively interested in the case $\mathcal{X}'= \mathcal{S}^{\prime \ast}_{\dagger}(\mathbb{R}^{n})$.

We call a cone $C'$ \emph{solid} if it is non-empty and $\operatorname*{int} C'\neq \emptyset$.

After this preliminaries, we are ready to present our Tauberian theorem. It is inverse to the ensuing Abelian statement that readily follows from the definition: If an ultradistribution $f\in\mathcal{S}'^{\ast}_{\dagger}[\Gamma]$ has quasiasymptotic behavior \eqref{eq:quasiasymptotic} in $\mathcal{S}'^{\ast}_{\dagger}(\mathbb{R}^{n})$, then 
\begin{equation}
			 \label{eq:generalTauberianTheoremquasi-LaplaceconvStrong}
			\lim_{r\to0^{+}} \frac{r^n}{\rho(1/r)}\laplace{f}{rz}=\laplace{g}{z} 
			\end{equation}
uniformly for $z$ in compact subsets of $T^{C}$. In the next theorem we write  $\sharp=(A_p),\:\{A_p\}$ for the Beurling and Roumieu cases of the specified weight sequence $A_p$. 	\begin{theorem}
		\label{t:quasiasympequiv} Assume that $M_p$ and $N_p$ both satisfy $(M.1)$ and $(M.2)$, while $M_p$ also satisfies $(M.3)'$ and $N_p$ satisfies $(M.1)^{\ast}$. Set $A_p=M_pN_p$.
		Let $f \in \mathcal{S}^{\prime\beurou}_{\sharp}[\Gamma]$ and let $\rho$ be regularly varying of index $\alpha$. 
Suppose that there is a non-empty solid subcone $C'\subset C$
such that for each $y\in C'$ the limit	 
\begin{equation}						\label{eq:generalTauberianTheoremquasi-Laplaceconv}
						\lim_{r \rightarrow 0^{+}} \frac{r^n}{\rho(1/r)}\laplace{f}{riy}
						\end{equation}	
exists. If there are $\omega\in C$ and $(\ell_{p})\in\mathfrak{R}^{\beurou}$  such that
\begin{equation}
				\label{eq:generalTauberianTheoremquasi-Bound}
			\limsup_{r\to0^{+}}\underset{\theta\in (0,\pi/2]}{\sup_{|x|^{2}+\sin^{2}\theta=1}}	\frac{r^n
			e^{- A_{\ell_{p}}^{*}\left(\frac{1}{\sin\theta} \right)}}{\rho(1/r)} \left| \laplace{f}{r(x+i\sin\theta \omega)} \right| <\infty,
			\end{equation}	
then $f$ has quasiasymptotic behavior with respect to $\rho$ in $\mathcal{S}^{\prime\beurou}_{\dagger}(\mathbb{R}^{n})$.

	\end{theorem}
	
	\begin{proof} 
In view of Corollary \ref{t:generalTauberianTheorem}, it suffices to show that the Laplace transform of $f$ satisfies a bound of the form
	\begin{equation}\label{eq1tauultraLaplace}
\frac{r^{n}}{\rho(1/r)}\left|\laplace{f}{r(x+i\sigma \omega)}\right|\leq L \exp\left(M_{\ell'_{p}}(|x|)+N_{\ell'_{p}}\left(\frac{1}{\sigma}\right)\right)
	\end{equation}
for some $(\ell'_{p})\in\mathfrak{R}^{\ast},$ $L,\sigma_0>0$ and all $x\in\mathbb{R}^{n}$ and $0<\sigma<\sigma_0$. We may assume $(M.2)$ holds for both  $L_pM_p$ and $L_pN^{\ast}_p$ (with constants $A$ and $H$). We can also assume that $L_p\geq 1$ for all $p$. Using \eqref{eq:generalTauberianTheoremquasi-Bound}, there are $0<r_0<1$ and $L_1$ such that for any $0<r<r_0$
\begin{equation}\label{eq2tauultraLaplace}
\frac{r^n}{\rho(1/r)} \left| \laplace{f}{r(x+i\sin\theta \omega)} \right| \leq L_1 \exp\left( A_{\ell_{p}}^{*}\left(\frac{1}{\sin\theta} \right)\right),
\end{equation}	
whenever $|x|^{2}+\sin^{2} \theta=1$, where we always keep $0<\theta<\pi/2$. On the other hand, applying Theorem \ref{t:boundednessLaplace} to the singleton $B=\{f\}$ and possibly enlarging $(\ell_p)$,
\begin{equation}\label{eq3tauultraLaplace}
\left| \laplace{f}{r(x+i\sigma \omega)} \right| \leq L_2 \exp\left( M_{\ell_p}(|x|) + A_{\ell_{p}}^{*}\left(\frac{1}{r\sigma} \right)\right),
\end{equation}
for any $0<r<1$, $x\in\mathbb{R}^{n}$ and $\sigma<r_0<1$. We may assume that $\rho(\lambda)=1$ for $\lambda<r_0$. Furthermore, Potter's estimate \cite[Theorem 1.5.4]{bingham1989regular} yields
\begin{equation}\label{eq4tauultraLaplace} 
		\frac{\rho(\lambda t)}{\rho(\lambda)} \leq L_{3} t^{\alpha} \max\{ t^{- 1}, t \} , \qquad t,\lambda>0.
	\end{equation}	
We keep arbitrary $r<1$, $x\in\mathbb{R}^{n}$, $0<\sigma<r_0$, and write $r'=\sqrt{|x|^{2}+\sigma^{2}}$, $x'=x/r'$, and $\sin \theta=\sigma/r'$. If $rr'<r_0$, we obtain from \eqref{eq4tauultraLaplace}, \eqref{eq2tauultraLaplace}, and the fact that $M_{\ell_p}(t)$ increases faster than $\log t$,
\begin{align*}
\frac{r^{n}}{\rho(1/r)}\left|\laplace{f}{r(x+i\sigma \omega)}\right|&\leq L_1 L_3 \left(\frac{1}{r'}\right)^{\alpha+n} \max\left\{r',\frac{1}{r'}\right\}\exp\left(A^{\ast}_{\ell_p}\left(\frac{r'}{\sigma}\right)\right)
\\
&= O\left(\exp\left(M_{\ell_p}(2|x|)+A^{\ast}_{\ell_p}\left(\frac{2|x|}{\sigma}\right)\right) \right).
\end{align*}
Similarly, if $rr'\geq r_0$, we employ \eqref{eq3tauultraLaplace}, \eqref{eq4tauultraLaplace}, $\rho(1/(rr'))=1$, and $(M.2)'$ for both $L_pM_p$ and $L_pA_p$ to conclude that for some $h'>0$
$$
\frac{r^{n}}{\rho(1/r)}\left|\laplace{f}{r(x+i\sigma \omega)}\right|\leq O\left(\exp\left(M_{\ell_p}(h'|x|)+A^{\ast}_{\ell_p}\left(\frac{h'|x|}{\sigma}\right)\right) \right).
$$
We have found in all cases
$$
\frac{r^{n}}{\rho(1/r)}\left|\laplace{f}{r(x+i\sigma \omega)}\right|\leq L_4\exp\left(M_{\ell_p}(h|x|)+A^{\ast}_{\ell_p}\left(\frac{h|x|}{\sigma}\right)\right) .
$$
for some $L_{4}$ and $h=\max\{h',2\}$. It remains to observe that 
$A^{\ast}_{\ell_p}\left(h|x|/\sigma\right)\leq M_{\ell_p}(h|x|)+ N^{\ast}_{\ell_p} (h/\sigma)$, so that   \eqref{eq1tauultraLaplace} holds with $\ell'_p=\ell_p/(Hh),$ by $(M.2)$.	
	\end{proof}

\section{Sharpening the bound \eqref{eq:boundednessLaplace1}} 
\label{Section Laplace isomorphism}
If the sequence $M_p$ and the cone $\Gamma$ satisfy stronger conditions, it turns out that the bound \eqref{eq:boundednessLaplace1} can be considerably improved. In fact, we shall show here how to remove the $\varepsilon$ term from \eqref{eq:boundednessLaplace1}.  
Recall $\Gamma$ is a solid cone if $\interior \Gamma \neq \emptyset$.

We start with three  lemmas, from which our improvement of Theorem \ref{t:boundednessLaplace} will follow.

	\begin{lemma}
		\label{l:holomorphicaslaplace}
		Let $\{F_{j}\}_{j \in I}$ be a family of holomorphic functions on $T^{C}$. Suppose that for some $(\ell_{p}) \in \mathfrak{R}^{\beurou}$ and each $\varepsilon > 0$ there is $L = L_{\varepsilon} > 0$ such that for all $j \in I$
			\begin{equation}
				\label{eq:holomorphicaslaplace}
				\left| (1 + |\re z|)^{n + 2} F_{j}(z) \right| \leq L \exp\left( \varepsilon |\im z| + N_{\ell_{p}}^{\beurou}\left(\frac{1}{\Delta_{C}(\im z)}\right) \right), \qquad z\in T^C.
			\end{equation}
		Then there are $(h_{p}) \in \mathfrak{R}^{\ast}$ and $f_{j} \in C^{1}(\reals^{n})$ with $\supp f_{j} \subseteq \Gamma$, $\forall j \in I$, such that $\{ e^{-N_{h_{p}}(|\cdot|)} f_{j} \}_{j \in I}$ is a bounded set in $L^{\infty}(\reals^{n})$ and $F_{j}(z)  = \laplace{f_{j}}{z}$, $j \in I$.
	\end{lemma}
	
	\begin{proof} We closely follow the proof of the lemma in \cite[Section~10.5, p.~148]{vladimirovbook}. We may assume that $L_{p} N_{p}$ satisfies $(M.1)$ and $(M.2)$. From (\ref{eq:holomorphicaslaplace}) it follows in particular that
			\[ (1 + |\cdot|) F_{j}(\cdot + i y) \in L^{1}(\reals^{n}),  \qquad \forall y \in C, j \in I. \] 
From the Cauchy formula we obtain for each compact subset $K\subset C$ and each $j \in I$
			\[ \sup_{y \in K} \left| \frac{\partial}{\partial y_{k}} F_{j}(x + iy) \right| = \bigoh{\frac{1}{(1 + |x|)^{n + 2}}}, \qquad k \in \{1, \ldots, n\}. \]
		Therefore,
			\[ g_{j}(\xi, y) = (2\pi)^{-n} e^{\xi \cdot y} \mathcal{F}\left\{ F_{j}(\cdot + iy); \xi \right\} \in C^{1}(\reals^{n} \times C) , \qquad j \in I , \]
			where $\mathcal{F}$ stands for the Fourier transform.
			Furthermore, for each $k \in \{1, \ldots, n\}$,
			\begin{equation*}
				\frac{\partial}{\partial y_{k}} g_{j}(\xi, y)
											  = (2 \pi)^{-n} e^{\xi \cdot y} \left[ \xi_{k} \mathcal{F}\left\{F_{j}(\cdot + iy);\xi\right\} + i \mathcal{F}\left\{\frac{\partial}{\partial x_{k}} F_{j}(\cdot + iy); \xi \right\} \right] = 0,
			\end{equation*}
so that the $C^1$ functions $f_{j}(\xi): = g_{j}(\xi, y)$ do not depend on $ y \in C$. By (\ref{eq:holomorphicaslaplace}), there is $L' = L'_{\varepsilon} > 0$ such that
			\begin{equation} 
				\label{eq:holomorphicaslaplace-intermediate}
				\left| f_{j}(\xi) \right| \leq L' \exp\left( \xi \cdot y + \varepsilon |y| + N_{\ell_{p}}^{*}\left(\frac{1}{\Delta_{C}(y)}\right) \right) , \qquad \xi \in \reals^{n},\: y \in C\: , j \in I .
			\end{equation}
		Take any $\xi_{0} \notin \Gamma$. As $(\Gamma^{*})^{*} = \Gamma$, there is some $y_{0} \in C$ such that $\xi_{0} \cdot y_{0} = -1$. Since $\Delta_{C}(\lambda y_{0}) = \lambda \Delta_{C}(y_{0})$ for $\lambda > 0$, we conclude from (\ref{eq:holomorphicaslaplace-intermediate}) for $\varepsilon = (2|y_{0}|)^{-1}$ and $y = \lambda y_{0}$ that
			\[ \left| f_{j}(\xi_{0}) \right| \leq L' \exp\left( - \frac{\lambda}{2} + N_{\ell_{p}}^{*}\left(\frac{1}{\lambda \Delta_{C}(y_{0})}\right) \right) , \qquad  \lambda > 0 . \]
		By letting $\lambda \rightarrow \infty$, it follows that this is only possible if $f(\xi_{0}) = 0$. We conclude that $\supp f_{j} \subseteq \Gamma$ for each $ j \in I$.  
		
		Now take an arbitrary $y_{0} \in C$ such that $|y_{0}| = 1$, then (\ref{eq:holomorphicaslaplace-intermediate}) gives us for $\varepsilon = 1/2$ and $y = \lambda y_{0}$, $\lambda > 0$, 
			\[ \left| f_{j}(\xi) \right| \exp\left( - (1 + |\xi|) \lambda - N_{\ell_{p}}^{*}\left(\frac{1}{\lambda \Delta_{C}(y_{0})}\right) \right) \leq L' e^{-\frac{\lambda}{2}} . \]
		We integrate this inequality with respect to $\lambda$ on $(0, \infty)$ in order to gain an estimate on the $f_{j}$. The 1-dimensional case of \cite[Lemma 5.2.6, p. 97]{carmichael2007boundary} applied to the open cone $(0, \infty)$, yields the existence of constants $L'', c > 0$ such that
			\[ \int_{0}^{\infty}  \exp\left( - (1 + |\xi|) \lambda - N_{\ell_{p}}^{*}\left(\frac{1}{\lambda \Delta_{C}(y_{0})}\right) \right) d\lambda \geq L'' \exp\left( - N_{\ell_{p}}(c (1 + |\xi|)) \right) . \]
		Hence, using \cite[Lemma 2.1.3, p. 16]{carmichael2007boundary}, it follows for any $\xi \in \reals^{n}$ and $j \in I$ that 
			\[ \left| f_{j}(\xi) \right| \leq \frac{2 L'}{L''} \exp\left( N_{\ell_{p}}(c(1 + |\xi|)) \right) \leq \frac{2 L'}{L''} \exp\left(N_{\ell_{p}}(2 c) + N_{\ell_{p}}(2 c |\xi|) \right) . \]
	 The proof is complete noticing that by the Fourier inversion $F_{j}(z) = \laplace{f_{j}}{z}$.
	\end{proof}
	
	Recall an ultrapolynomial \cite{ultradistributions1} of type $\beurou$ is an entire function 
	\[ P(z) = \sum_{m = 0}^{\infty} a_{m} z^{m} , \qquad a_{m} \in \complexes , \]
where the coefficients satisfy $|a_{m}| \leq L / H_{m} M_{m}$ for some $(h_{p}) \in \mathfrak{R}^{\beurou}$ and $L>0$. It is ensured by condition $(M.2)$ that the multiplication of two ultrapolynomials is again an ultrapolynomial (cf. \cite[Proposition~4.5, p.~58]{ultradistributions1} and \eqref{eqequivalencegrowth1}).

	\begin{lemma}
		\label{l:ultrapolynomial} Let $\Gamma$ be a solid cone \emph{solid} and let $(\ell_{p}) \in \mathfrak{R}^{\beurou}$. Suppose that $L_pM_p$ satisfies $(M.1)$, $(M.2)$, and $(M.3)$. Then, there are an ultrapolynomial $P$ of type $\beurou$ and constants $L, L' \geq 1$ such that
			\begin{equation}
				\label{eq:ultrapolylowerbound}
				e^{M_{\ell_{p}}(|z|)} \leq \left| P(z) \right| \leq L' e^{M_{\ell_{p}}(L |z|)} , \qquad \forall z \in T^{C} . 
			\end{equation}
	\end{lemma}
	
	\begin{proof} Set
			\[ \widetilde{P}(z) := \prod_{p=1 }^{\infty} \left( 1 + \frac{z}{\ell_{p} m_{p}} \right) , \qquad  z \in \complexes , \]
an ultrapolynomial of type $\beurou$ satisfying a bound $P(z)=O(e^{M_{\ell_p}(L''|z|)})$ \cite[Proposition~4.5 and Proposition~4.6, pp.~58--59]{ultradistributions1}. Now for $\re z \geq 0$ as in \cite[p.~89]{ultradistributions1}
			\[ \left| \widetilde{P}(z) \right| \geq \sup_{p \in \naturals} \prod_{q = 1}^{p} \frac{|z|}{\ell_{q} m_{q}} = \sup_{p \in \naturals} \frac{M_{0} |z|^{p}}{L_{p} M_{p}} = e^{M_{\ell_{p}}(|z|)} . \]
		Since we assumed $\interior \Gamma \neq \emptyset$, there is a basis $\{ e_{1}, \ldots, e_{n} \}$ of $\reals^{n}$ such that $e_{j} \in \interior \Gamma$ for $1 \leq j \leq n$. Find also $\lambda>0$ such that $\lambda\min_{j}|e_{j}\cdot z|\geq |z|$ for all $z\in C$. Now define
			\[ P(z) := \prod_{j = 1}^{n} \widetilde P(- \lambda n^{1/2} i e_{j} \cdot z) , \]
	which is an ultrapolynomial of type $\beurou$ as well and the upper bound in \eqref{eq:ultrapolylowerbound} holds because of \eqref{eq:M2} applied to $M_{\ell_p}$. Since for any $z \in T^{C}$ we have $\re(- n i e_{j} \cdot z)> 0$, $1 \leq j \leq n$, one then obtains for any $z$ in the tube domain
			\[ \left| P_{\ell_{p}}(z) \right| \geq \exp\left( \sum_{j = 1}^{n} M_{\ell_{p}}(n^{1/2}\lambda |e_{j} \cdot z|) \right) \geq \exp\left( M_{\ell_{p}} \left( |z| \right) \right) . \]
			\end{proof}
	
	\begin{lemma}
		\label{l:boundonsupdiff} Let $(\ell_p)\in\mathfrak{R}^{\ast}$.
		It holds that for any $y \in C$
			\begin{equation}
				\label{eq:boundonsupdiff}
				\sup_{\xi \in \Gamma} \exp\left(N_{\ell_p}( |\xi|) - y \cdot \xi\right) \leq \exp\left(N^{*}_{\ell_p}\left(\frac{1}{\Delta_{C}(y)}\right)\right) . 
			\end{equation}
	\end{lemma}
	
	\begin{proof} We only make use of $(M.1)^{\ast.}$
		Using the estimate (\ref{eq:dotestimate}), we obtain for any $y \in C$
			\[ \sup_{\xi \in \Gamma} e^{N_{\ell_p}( |\xi|) - y \cdot \xi} \leq \sup_{t \geq 0} e^{N_{\ell_p}(t) - \Delta_{C}(y) t} , \]
		so that (\ref{eq:boundonsupdiff}) follows from \cite[Lemma~5.6]{Vogt1984} 
			\[ \sup_{t > 0} \left\{ N_{\ell_p}(t) - st \right\} \leq N_{\ell_p}^{*}\left(\frac{1}{s}\right) , \qquad  s > 0 . \]
	\end{proof}
	
	\begin{theorem}
		\label{t:boundednessLaplaceStronger}		Suppose that the cone $\Gamma$ is solid, $M_p$ and $N_p$ both satisfy $(M.1)$ and $(M.2)$, and $M_p$ also satisfies $(M.3)$.   Then, a set $B\subset \sss^{\prime \beurou}_{\dagger}[\Gamma]$ is bounded if and only if there are $L>0$ and $(\ell_{p}) \in \mathfrak{R}^{\beurou}$ such that for all $f \in B$
			\begin{equation}
				\label{eq:boundednessLaplaceStronger}
				\left| \laplace{f}{z} \right| \leq L \exp\left( M_{\ell_{p}}(|z|) + N_{\ell_{p}}^{*}\left(\frac{1}{\Delta_{C}(\im z)} \right) \right) , \qquad z \in T^{C} . 
			\end{equation}
	\end{theorem}
	
	\begin{proof} We only need to show that if $B=\{f_j\}_{j\in I}$ is bounded then \eqref{eq:boundednessLaplaceStronger} holds.  By Theorem \ref{t:boundednessLaplace}, there is $(\ell_{p}) \in \mathfrak{R}^{\ast}$ such that for any $\varepsilon > 0$ there is $L = L_{\varepsilon} > 0$ such that for all $j \in I$
			\[ \left| \laplace{f_{j}}{z} \right| \leq L \exp\left( \varepsilon |y| + M_{\ell_{p}}(|z|) + N_{\ell_{p}}^{*}\left(\frac{1}{\Delta_{C}(y)} \right) \right) , \qquad \forall z \in T^{C} . \]
We may assume $L_{p} M_{p}$ satisfies $(M.1)$, $(M.2)$ and $(M.3)$. Let $P$ be the ultrapolynomial constructed as in Lemma \ref{l:ultrapolynomial}. Fix $k\geq H^{n+2}$, where $H$ is the constant occurring in $(M.2)'$ for $L_pM_p$. We consider the ultrapolynomial $Q(z)=P(kz)$, so that it satisfies the bounds $e^{M_{\ell_{p}}(k|z|)} \leq |Q(z)| \leq L'e^{M_{\ell_{p}}(\nu |z|)}$ for all $ z \in T^{C}$ and some $\nu > 0$. 
Set now $F_{j}(z) = \laplace{f_{j}}{z}$, which are holomorphic functions on $T^{C}$. In view of \eqref{eq:M2prime}, the family $\{F_{j} / Q\}_{j \in I}$ satisfies the conditions of Lemma \ref{l:holomorphicaslaplace}, so that there are $g_{j} \in C^{1}(\reals^{n})$ with $\supp g_{j} \subseteq \Gamma$ for which there is some $(\ell'_{p}) \in \mathfrak{R}^{\beurou}$ such that $\{ \exp(- N_{\ell'_{p}}(|\cdot|)) g_{j} \}_{j \in I}$ is a bounded subset of $L^{\infty}(\reals^{n})$ and $F_{j}(z) = Q(z) \laplace{g_{j}}{z}$ for each $j \in I$. Now, taking into account \eqref{eq:M2prime} (we may assume $H$ is the same constant for both $L_pM_p$ and $L'_pN_p$) and Lemma \ref{l:boundonsupdiff}, there are some $L'',L''' > 0$ such that for all $j \in I$
			\begin{align*} 
				\left| F_{j}(z) \right| &\leq L'' e^{M_{\ell_{p}}(\nu |z|)} \int_{\Gamma} e^{N_{\ell'_{p}}(|\xi|) - y \cdot \xi} d\xi \\ 
							     &\leq L''' e^{M_{\ell_{p}}(\nu |z|)} \sup_{\xi \in \Gamma} e^{N_{\ell'_{p}}(k |\xi|) - y \cdot \xi} \int_{\Gamma}\frac{d\xi} {(1+|\xi|)^{n+2}} \\
							     &\leq A L' \left( \int_{\Gamma}\frac{d\xi} {(1+|\xi|)^{n+2}} \right) \exp\left( M_{\ell_{p}}(\nu |z|) + N_{\ell'_{p}}^{*}\left(\frac{k}{\Delta_{C}(y)}\right) \right)  .
			\end{align*}
Hence, we obtain a bound of type (\ref{eq:boundednessLaplaceStronger}) for the sequence $k_{p} = \min\{ \ell_{p} / \nu , \ell'_{p} / k \}$.
	\end{proof}
Theorem \ref{t:boundednessLaplaceStronger} can be used to draw further topological information. In fact, it leads to an isomorphism between $\mathcal{S}^{\prime \ast}_{\dagger}[\Gamma]$ and analogs of the Vladimirov algebra \cite[Chapter~12]{vladimirovbook} $H(T^{C})$ of holomorphic functions on $T^{C}$. Given $\ell>0$, we define the Banach space $\mathcal{O}_{\ell}(T^{C})$ of all holomorphic functions $F$ on the tube domain $T^{C}$ that satisfy the bounds
\[
\|F\|_{\ell}=\sup_{z\in T^{C}}|F(z)|e^{-M(\ell |z|)-N^{\ast}\left(\frac{\ell}{\Delta_{C}(\im z)}\right)}<\infty.
\] 
We then introduce the $(DFS)$ and $(FS)$ spaces 
\[
\mathcal{O}^{(M_p)}_{(N_p)}(T^{C})=\varinjlim_{\ell} \mathcal{O}_{\ell}(T^{C}) \qquad \mbox{and}\qquad \mathcal{O}^{\{M_p\}}_{\{N_p\}}(T^{C})=\varprojlim_{\ell} \mathcal{O}_{\ell}(T^{C}).
\]

The arguments we have given above actually show that the Laplace transform maps $\mathcal{S}^{\prime \ast}_{\dagger}[\Gamma]$ bijectively into  $\mathcal{O}^{\ast}_{\dagger}(T^{C})$ and that this mapping and its inverse map bounded sets into bounded sets (cf. the property \eqref{eqequivalencegrowth1} in the Roumieu case). Since the spaces under consideration are all bornological, we might summarize the results from this section as follows,

\begin{theorem}\label{Laplace isomorphism ultradistributions} Let $\Gamma$ be a solid convex acute cone and suppose that $M_p$ and $N_p$ both satisfy $(M.1)$ and $(M.2)$, while $M_p$ also satisfies $(M.3)$. Then, the Laplace transform
\[ \mathcal{L}:\mathcal{S}^{\prime \ast}_{\dagger}[\Gamma]\to \mathcal{O}^{\ast}_{\dagger}(T^{C})
\]
is an isomorphism of locally convex spaces. 
\end{theorem}

\smallskip

\subsection*{Acknowledgement}
The authors would like to thank Andreas Debrouwere for useful discussions.


\begin{thebibliography}{99}

\bibitem{BerceanuGheorghe}
{\sc S.~Berceanu and A.~Gheorghe}, {\em On the asymptotics of distributions with support in a cone},
J. Math. Phys. 26 (1985), 2335--2341.

\bibitem{bingham1989regular}
{\sc N.~Bingham, C.~Goldie and J.~Teugels}, {\em Regular variation},
  Encyclopedia of Mathematics and its Applications, 27, Cambridge University Press, Cambridge, 1987.
  
  
\bibitem{carmichael2007boundary}
{\sc R.~D.~Carmichael, A.~J.~Kami{\'n}ski, and S.~Pilipovi{\'c}}, {\em Boundary values and convolution in ultradistribution spaces},
Series on Analysis, Applications and Computation, 1, World Scientific Publishing Co. Pte. Ltd., Hackensack, NJ, 2007.

\bibitem{D-V-V2018} {\sc A.~Debrouwere, H.~Vernaeve, and J.~Vindas,} {\em Optimal embeddings of ultradistributions into differential algebras,} Monatsh. Math. 186 (2018), 407--438.

\bibitem{D-V-indlimultra}{\sc A.~Debrouwere and J.Vindas,} {\em On weighted inductive limits of spaces of ultradifferentiable functions and their duals,} Math. Nachr. 292 (2019), 573--602.

\bibitem{D-V-Roumieuind} {\sc A.~Debrouwere and J.~Vindas,} {\em A projective description of generalized Gelfand-Shilov spaces of Roumieu type,} in: Analysis, Probablility, Applications, and Computation, pp. 407--417. Proceedings of the 11th ISAAC congress (V\"{a}xj\"{o}, 2017). Trends in Mathematics, Birkh\"{a}user, Cham, 2019.

\bibitem{dimovski2016}
{\sc P.~Dimovski, S.~Pilipovi{\'c}, B.~Prangoski, and J.~Vindas}, {\em Convolution of ultradistributions and ultradistribution spaces associated to translation-invariant Banach spaces},
Kyoto J. Math. 56 (2016), 401--440. 


\bibitem{drozhzhinov2016} {\sc Yu.~N.~Drozhzhinov,} {\em Multidimensional Tauberian theorems for generalized functions,} Uspekhi Mat. Nauk 71 (2016), 99--154; English translation in
Russian Math. Surveys 71 (2016), 1081--1134. 

\bibitem{d-z1998} {\sc Yu.~N.~Drozhzhinov and B.~I.~Zav'yalov,}
{ \em A Tauberian theorem of Wiener type for generalized functions of slow growth,}
Mat. Sb. 189 (1998), 91--130; English translation in
Sb. Math. 189 (1998), 1047--1086. 

\bibitem{d-z2003} {\sc Yu.~N.~Drozhzhinov and B.~I.~Zav'yalov,} {\em Multidimensional Tauberian theorems for generalized functions with values in Banach spaces,} Mat. Sb. 194 (2003), 17--64; English translation in Sb. Math. 194 (2003), 1599--1646. 

\bibitem{d-z2008} {\sc Yu.~N.~Drozhzhinov and B.~I.~Zav'yalov,} {\em Applications of Tauberian theorems in some problems of mathematical physics,} Teoret. Mat. Fiz. 157 (2008), 373--390; English translation in Theoret. and Math. Phys. 157 (2008), 1678--1693.

\bibitem{estrada2012distributional}
{\sc R.~Estrada and R.~Kanwal}, {\em A distributional approach to asymptotics.
 Theory and applications}, Birkh{\"a}user Boston, Boston, MA, 2002.


\bibitem{GrochenigZimmermann}
 {\sc K.~Gr{\"o}chenig and G.~Zimmermann}, {\em Spaces of test functions via the STFT},
J. Funct. Spaces Appl. 2 (2004), 25--53.

\bibitem{ultradistributions1}
{\sc H.~Komatsu}, {\em Ultradistributions. {I}. Structure theorems and a
  characterization}, J. Fac. Sci. Univ. Tokyo Sect. IA Math. 20 (1973), 25--105.
  
\bibitem{ultradistributions3}
{\sc H.~Komatsu}, {\em Ultradistributions, III. Vector valued ultradistributions and the theory of kernels},  J. Fac. Sci. Univ. Tokyo Sect. IA Math. 29 (1982), 653--717.
  
\bibitem{K-P-S-V2016} {\sc S.~Kostadinova, S.~Pilipovi\'{c}, K.~Saneva, and J.~Vindas,} {\em The short-time Fourier transform of distributions of exponential type and Tauberian theorems for S-asymptotics,} Filomat 30 (2016), 3047--3061.  
  
\bibitem{structquasiultra}
{\sc L.~Neyt and J.~Vindas}, {\em Structural theorems for quasiasymptotics of ultradistributions}, Asymptot. Anal. 114 (2019), 1--18.
  
  
\bibitem{Vogt1984}
{\sc H.-J.~Petzsche and D.~Vogt}, {\em Almost analytic extension of ultradifferentiable functions and the boundary values of holomorphic functions}, Math. Ann. 267 (1984), 17--35.

  
\bibitem{P-P-Vconv} {\sc S.~Pilipovi\'{c}, B.~Prangoski, and J.~Vindas,} {\em On quasianalytic classes of Gelfand-Shilov type. Parametrix and convolution,} J. Math. Pures Appl. 116 (2018), 174--210.  
  
\bibitem{10.2307/44095755}
{\sc S.~Pilipovi{\'c} and B.~Stankovi{\'c}}, {\em Quasi-asymptotics and
  S-asymptotics of ultradistributions},  Bull. Acad. Serbe Sci. Arts Cl. Sci. Math. Natur. 20 (1995), 61--74.
  
 \bibitem{pilipovic-stankovic1997} {\sc S.~Pilipovi\'{c} and B.~Stankovi\'{c}}, {\em Tauberian theorems for integral transforms of distributions,} Acta Math. Hungar. 74 (1997), 135--153.  
  
  
\bibitem{stevan2011asymptotic}
{\sc S.~Pilipovi{\'c}, B.~Stankovi{\'c}, and J.~Vindas}, {\em Asymptotic
  behavior of generalized functions}, Series on Analysis, Applications and
  Computation, 5, World Scientific Publishing Co. Pte. Ltd., Hackensack, NJ, 2012.
  

  
  \bibitem{P-V2014}  {\sc S.~Pilipovi{\'c} and J.~Vindas}, {\em Multidimensional Tauberian theorems for vector-valued distributions,} Publ. Inst. Math. (Beograd) 95 (2014), 1--28.
  
  \bibitem{P-V2019} {\sc S.~Pilipovi{\'c} and J.~Vindas}, {\em Tauberian class estimates for vector-valued distributions,} Mat. Sb. 210 (2019), 115--142; English version in  Sb. Math. 210 (2019), 272--296.
  
\bibitem{BojangLaplaceUltradistr}
{\sc B.~Prangoski}, {\em Laplace transform in spaces of ultradistributions},
Filomat 27 (2013), 747--760.
  

\bibitem{vladimirov1976}{\sc V.~S.  Vladimirov,}
{\em A multidimensional  generalization  of  the Tauberian  theorem  of
Hardy  and
Littlewood,}  Izv. Akad. Nauk SSSR Ser. Mat. 40 (1976), 1084--1101; English translation in Math. USSR-Izv. 10 (1976), 1031--1048.
  
 \bibitem{vladimirovbook} {\sc V. ~S.~Vladimirov,} {\em Methods  of the theory of generalized functions,}
Analytical  Methods  and Special Functions, 6, Taylor \& Francis, London, 2002. 
  
 \bibitem{vladimirov1988tauberian}
{\sc V.~S.~Vladimirov, Yu.~N.~Drozhzhinov, and B.~I.~Zav'yalov}, {\em Tauberian
  theorems for generalized functions}, Kluwer Academic Publishers, Dordrecht, 1988.
  

 \bibitem{Vladimirov1979}    {\sc V.~S.~Vladimirov and B.~I.~Zav'yalov}, {\em Tauberian theorems in quantum field theory}, Teoret. Mat. Fiz. 40 (1979), 155--178; English translation in Theoret. and Math. Phys. 40 (1979), 660--677.
  
  
\bibitem{Yakimiv} {\sc A.~L.~Yakimiv,}
 {\em Probabilistic  applications  of  Tauberian  theorems,} Modern  Probability  and Statistics, VSP, Leiden, 2005.

\bibitem{Zavyalov}
{\sc B.~I. Zav'yalov}, {\em Scaling  of  electromagnetic  form-factors  and  the  behavior of their  Fourier
transforms in the neighborhood of the light cone}, Teoret. Mat. Fiz. 17 (1973), 178--188; English translation in Theoret. and Math. Phys. 17 (1973), 1074--1081. 

\end{thebibliography}
\end{document}